\documentclass[a4paper,12pt]{amsart} 
\usepackage{amsmath,amsthm,amssymb}  

\usepackage{graphicx}

\numberwithin{equation}{section}

\theoremstyle{plain}
\newtheorem{theorem}{Theorem}[section]
\newtheorem{proposition}[theorem]{Proposition}         
\newtheorem{corollary}[theorem]{Corollary} 
\newtheorem{lemma}[theorem]{Lemma} 
   
\theoremstyle{definition}  
 
\newtheorem{remark}{Remark}  
 
\newtheorem{notation}{Notation}

\newcommand{\C}{\mathbb C}   
\newcommand{\R}{\mathbb R}
\newcommand{\Z}{\mathbb Z}

\renewcommand{\P}{\mathbb P}

\newcommand{\su}{{\mathfrak s\mathfrak u}}
\renewcommand{\sl}{{\mathfrak s\mathfrak l}}

\newcommand{\al}{\alpha}
\newcommand{\be}{\beta} 
\newcommand{\ga}{\gamma}
\newcommand{\de}{\delta}
\newcommand{\la}{\lambda}
\newcommand{\si}{\sigma} 
 
\newcommand{\La}{\Lambda}
\newcommand{\eps}{\epsilon}

\DeclareMathOperator{\tr}{tr}

\DeclareMathOperator{\diag}{diag}

\newcommand{\calA}{\mathcal{A}}
\newcommand{\calB}{\mathcal{B}}
\newcommand{\calH}{\mathcal{H}}

\newcommand{\sub}{\subseteq}  
\newcommand{\es}{\emptyset}

\newcommand{\st}{\ \vert\ }   
\renewcommand{\ll}{\lq\lq}
\newcommand{\rr}{\rq\rq\ }
\newcommand{\rrr}{\rq\rq}

\renewcommand{\b}{\partial}

\newcommand{\bp}{\begin{pmatrix}} 
\newcommand{\ep}{\end{pmatrix}}

\renewcommand{\d}{ {T} }
\newcommand{\X}{{E}}

\newcommand{\bs}{
\left(
\begin{smallmatrix}
} 
\renewcommand{\es}{
\end{smallmatrix}
\right)
} 

\newcommand{\bsq}{
\left[
\begin{smallmatrix}
} 
\newcommand{\esq}{
\end{smallmatrix}
\right]
} 

\newcommand{\zbar}{  {\bar z}  }

\newcommand{\w}{  
\left(
\begin{smallmatrix}
 & \la \\
 \!-1\!/\!\la\  & 
\end{smallmatrix}
\right)
  }
  
  \newcommand{\bigw}{  
\left(
\begin{matrix}
 & \la \\
 \!-1\!/\!\la\  & 
\end{matrix}
\right)
  }

\newcommand{\sqrta}{
\text{\scriptsize$\sqrt a$}
}

\newcommand{\sqrtma}{
\text{\scriptsize$\sqrt {-a}$}
}

\newcommand{\sqrtaa}{
\text{\scriptsize$\sqrt{a+t+\bar t
\vphantom{t^t}
}$}
}

\newcommand{\sqrtmaa}{
\text{\scriptsize$\sqrt{-a-t-\bar t
\vphantom{t^t}
}$}
}

\makeatletter

\newcommand{\Rmnum}[1]
{\expandafter\@slowromancap\romannumeral #1@}
\makeatother
\renewcommand{\P}{P\Rmnum3 }
\newcommand{\PP}{P\Rmnum3}

\begin{document}     

\title[The $tt^\ast$ structure of $\C P^1$]{The $tt^\ast$ structure of the quantum cohomology of $\C P^1$ from the viewpoint of differential geometry}  

\author{Josef F. Dorfmeister, Martin A. Guest and Wayne Rossman}      

\date{}   

\maketitle 

\section{Introduction}\label{intro}

The quantum cohomology of $\C P^1$ provides a distinguished solution of the third Painlev\'e (\PP) equation.  S. Cecotti and C. Vafa  discovered this from a physical viewpoint (see \cite{CeVa91}, \cite{CeVa92}).  We shall derive this from a differential geometric viewpoint, using the theory of harmonic maps and in particular the generalized Weierstrass representation  (DPW representation) for surfaces of constant mean curvature.  The nontrivial aspects are the characterization of the solution, and its global behaviour.  As yet, no treatment (including ours) could be described as completely satisfactory, but we hope that our viewpoint provides additional insight. 

Mirror symmetry provides the context for this example:  whereas the quantum cohomology of a Calabi-Yau manifold corresponds to a variation of Hodge structure,  the quantum cohomology of a Fano manifold (such as $\C P^1$) should correspond to a variation of 
\ll semi-infinite Hodge structure\rr or
\ll non-commutative Hodge structure\rr (see \cite{Bar01}, \cite{He03}, \cite{KaKoPa07}).   
In both cases, the variation of Hodge structure can be described as a 
\ll $tt^\ast$ structure\rrr.  This originates from the physical notion of the ground state metric (Zamolodchikov metric) on a moduli space of supersymmetric field theories. It represents a fusion of topological (holomorphic) and anti-topological (anti-holomorphic) objects.
In differential geometric terms, a $tt^\ast$ structure
can be described as a certain kind of pluriharmonic map.   

In the language introduced by C. Hertling (see \cite{He03} and sections 10 and 11 of \cite{HeSe07}) the essential point is that the  quantum cohomology of $\C P^1$ gives rise to a TERP-structure which is pure and polarized.  Such structures arise naturally in singularity theory, and an independent approach to the pure and polarized property for the mirror partner of $\C P^1$ (and other Fano manifolds) has been given by C. Sabbah in \cite{Sa08}.  H. Iritani (\cite{IrXX}) described the result of Cecotti and Vafa for $\C P^1$ much more explicitly,  from the mirror symmetry viewpoint. Our approach constitutes yet another formulation: it says that the extended harmonic map remains entirely within a single Iwasawa orbit of the loop group $\La SU_{1,1}$. 

We shall now sketch in more concrete terms the necessary background information.  First of all,  it is well known that the (small) quantum cohomology of $\C P^1$ is a commutative algebra $\C[b,q]/(b^2-q)$ which specializes to the ordinary cohomology algebra
$\C[b]/(b^2)$  of $\C P^1$ when the value of the complex parameter $q$ is set equal to zero.  The quantum differential equation of $\C P^1$ is that given by the linear ordinary differential operator $(\la\b)^2 - q$, where $\la$ (often denoted by $\hbar$)  is a complex parameter and $\b=q\b/\b q$.  This can be regarded as a (consistent) linear system of two first-order operators,  which in turn can be regarded as a flat connection in the trivial bundle
\begin{align*}
\C\times\C^2\cong
H^2(\C P^1;\C)\times H^\ast(\C P^1;\C) \to
H^2(\C P^1;\C)
\cong\C
\end{align*}
(with a singularity at $q=0$).
In a suitable gauge, this can be identified with a connection whose flatness expresses the classical Gauss-Codazzi equations of a surface in (real) $3$-space;  to be precise, a spacelike surface of constant mean curvature (CMC) in Minkowski space $\R^{2,1}$.  Thus, the quantum cohomology of $\C P^1$ corresponds to a surface,  and we shall explain what this surface is.  

Of course this particular relation between the quantum cohomology of $\C P^1$ and a CMC surface in $\R^{2,1}$ is very special.  But it is a general principle (see \cite{Gu08}) that the quantum cohomology of any manifold corresponds to a pluriharmonic map into a symmetric space, and pluriharmonic maps into symmetric spaces may be treated by the same  loop group formalism.  The fact that a  $tt^\ast$ structure is a particular kind of pluriharmonic map was first observed by B. Dubrovin (\cite{Du93}); this was well known from the work of P. Griffiths for those $tt^\ast$ structures given by variations of polarized Hodge structure. Thus, pluriharmonic maps are the fundamental differential geometric objects here. In the case of 
$\C P^1$,  the pluriharmonic map is the Gauss map of the surface, and this Gauss map is a harmonic map into the hyperbolic disk.  

The quantum cohomology data mentioned above is holomorphic. In the theory of pluriharmonic maps it appears as the (normalized) potential in the DPW generalized Weierstrass representation, which is a holomorphic $1$-form with values in a complex loop algebra.  An appropriate choice of (non-holomorphic) gauge is necessary in order to obtain the Gauss-Codazzi equations of a surface.  This amounts to a choice of a real form of the loop algebra, and the obvious choice is the one associated to quantum cohomology with real coefficients.  However, there is still some ambiguity in the associated surface;  one obtains a family of surfaces related by \ll dressing transformations\rrr.  (This point is explained somewhat differently in section 11 of \cite{HeSe07} and in \cite{IrXX}.  From the loop group point of view there is a natural real structure of quantum cohomology.  However, from the variation of Hodge structure point of view, the dressing transformation ambiguity can be interpreted as an ambiguity of real structure.  This will be discussed more precisely in section \ref{comments}.)

The existence of a family of local $tt^\ast$ structures for  the quantum cohomology of $\C P^1$ follows easily from the Iwasawa decomposition.   However, the
observation of Cecotti and Vafa is that there is a distinguished global $tt^\ast$ structure whose domain of definition is maximal.  From a mathematical point of view this is surprising and nontrivial, but it can be established by brute force in the case of $\C P^1$, because the Gauss-Codazzi equations reduce to the \P equation, whose solutions have been studied deeply (see\footnote{A more direct proof has been given recently in \cite{GuLiXX}. This method also applies to the quantum cohomology of $\C P^n$ for $1\le n\le 4$ and several weighted projective spaces.
} \cite{MTW77}, \cite{FIKN06}).  Cecotti and Vafa single out this solution by physical arguments, and  Iritani obtains it by using K-theory and mirror symmetry.
It seems likely that a complete explanation of this  phenomenon will be of equal interest in surface theory, as the relation between the global properties of a CMC surface and its DPW potential is an active area of research.

For general background information on the DPW 
representation in surface theory we refer to \cite{DoPeWu98}, \cite{Do08}.  A survey of the loop group approach to harmonic maps can be found in \cite{Gu97}, and its relation to 
quantum cohomology in \cite{Gu08}.

The authors are very grateful to Claus Hertling and Hiroshi Iritani for discussing and explaining their work.  They thank Alexander Its for supplying the isomonodromy deformation arguments referred to in the proof of Theorem \ref{cv}, and for informing them about the paper \cite{ItNiXX}.  They thank Nick Schmitt for his guidance in creating the images in section \ref{global}. All three authors were partially supported by grants from the Japan Society for the Promotion of Science, and the second author also by the Alexander von Humboldt Foundation, at various stages of this research.  Comments from the referee were also very much appreciated.

\section{Spacelike CMC surfaces in $\R^{2,1}$}\label{cmc}

In this section we review the integrable systems approach to classical surface theory.  First we sketch some standard surface theory and the 
DPW generalized Weierstrass representation.  The case of spacelike CMC surfaces in $\R^{2,1}$ which we need here (see 
\cite{In00}, \cite{KoXX}, \cite{BrRoSc10}) is almost identical to the better known version for CMC surfaces in $\R^3$
(which can be found, for example, in \cite{Do08}).  

\noindent{\em Classical surface theory and the DPW representation}

We use the notation of section 3 of \cite{BrRoSc10}.
Let $f:U\to \R^{2,1}$ be a spacelike surface, where $U$ denotes an open subset of $\R^2=\C$  and $\R^{2,1}$ denotes $\R^3$ with the Minkowski inner product $(a,b)=a_1b_1+a_2b_2-a_3b_3$.  Spacelike means that the
induced metric on $U$ is positive definite.

In conformal coordinates $z=x+iy$, the classical surface data can be written
\begin{gather*}
g=4e^{2u}(dx^2 + dy^2)\\
H=\tfrac18 e^{-2u} (f_{xx} + f_{yy}, N)\\
Q=(N,f_{zz})
\end{gather*}
where $N$ is a unit normal field; $g$ is the induced metric, $H$  the mean curvature, and $Qdz^2$  the Hopf differential.
This data satisfies the Gauss-Codazzi equations
\begin{gather*}
u_{z\zbar} - H^2e^{2u}+\tfrac14\vert Q\vert^2 e^{-2u}=0
\\
Q_{\zbar}-2e^{2u}H_z=0.
\end{gather*}
Conversely, it is known that any solution of these equations defines a surface, up to rigid motion.

The CMC condition is $H= \text{constant}$, and then the second equation just says that $Q$ is holomorphic. The first equation has two remarkable properties: (a) when $H\ne 0$, it can be transformed (away from umbilic points, i.e.\ zeros of $Q$) into the sinh-Gordon equation; (b) from any spacelike CMC surface we obtain an $S^1$-family of spacelike CMC surfaces, because the first equation is the same when $Q$ is multiplied by a unit complex number.  The appearance of this \ll spectral parameter\rrr, and \ll soliton equations\rr such as the sinh-Gordon equation,
leads to the modern approach to surface theory which emphasizes loop groups as infinite-dimensional symmetry groups. 

\noindent{\em The zero curvature formulation}

Given a spacelike surface $f:U\to \R^{2,1}$, we obtain
at each point of $U$ an element of the natural symmetry group $SO_{2,1}$ by choosing a framing consisting of two orthonormal tangent vectors (by definition, spacelike) and a unit normal vector (timelike).
For calculations it is convenient to replace $SO_{2,1}$ by the locally isomorphic group $SU_{1,1}$.  We can regard the 
$S^1$-family of framings described above as a map  $F:U\to \La SU_{1,1}$ (called  an extended frame), where $\La SU_{1,1}$ (the loop group of $SU_{1,1}$) is the set of all smooth maps from $S^1$ to $SU_{1,1}$.  
Using the notation of \cite{BrRoSc10}, direct calculation gives 
$F^{-1}dF = \calA dz+\calB d\zbar$, where
\begin{equation}\label{MC}
\calA=
\tfrac12
\bp
u_z & - \tfrac1\la  2 i H e^u \\
\tfrac1\la  i Q e^{-u} & -u_z
\ep,
\ \ 
\calB=
\tfrac12
\bp
-u_{\bar z} & - \la   i \bar Q e^{-u} \\
2  i \la H e^{u} & u_{\bar z}
\ep,
\end{equation}
and $\la\in S^1$.
Given a basepoint $z_0\in U$, we can normalize $F$ (pre-multiply by an element of $\La SU_{1,1}$) so that
$F(z_0)=I$.  

We are regarding $SU_{1,1}$ as the real form of $SL_2\C$ given by the conjugate-linear involution
\[
C(A)=D\, ( {{\bar A}^t})^{-1} \, D,
\quad  D=\diag(1,-1),
\]
that is, $SU_{1,1}=
(SL_2\C)_C=\{A\in SL_2\C \st C(A)=A\}$. Similarly,
$\La SU_{1,1} = (\La SL_2\C)_C$ where
\[
C(\ga)(\la)=D(\overline{\ga(1/\bar\la)}^t)^{-1} D. 
\]
The twisted loop groups 
$(\La SL_2\C)_\si, (\La SU_{1,1})_\si$ are the subgroups of 
the loop groups $\La SL_2\C, \La SU_{1,1}$
defined by imposing the condition 
$\si(\ga)(\la)=\ga(-\la)$, i.e.\ they are the fixed points of
the involution
\[
\si(\ga)(\la)=
D\ 
\ga(-\la)
\ D.
\]
With this terminology, the map $F$ takes values in $(\La SU_{1,1})_\si$, and
$F^{-1}dF$ is a $1$-form with values in the twisted real loop algebra
$(\La \su_{1,1})_\si$.   

Conversely, let  $\al=\calA dz + \calB d\bar z$ be any $1$-form  on a {\em simply-connected} domain $U_0$ with values in 
$(\La \su_{1,1})_\si$ such that $\calA$ is linear in $1/\la$ and $\calB$ is linear in $\la$, and such that $d\al +\al\wedge\al=0$. 
The zero curvature condition $d\al +\al\wedge\al=0$ implies that there exists a map 
$F:U_0\to (\La SU_{1,1})_\si$ such that $\al=F^{-1}dF$. This $F$ is unique if we insist that 
$F(z_0)=I$.   The other conditions imply (see the end of this section)  that $\calA,\calB$ can be expressed in the above explicit
 form, for some $u, Q, H$. 
 By regarding $F\vert_{\la=1}$ as an orthonormal frame,  and then integrating,  we obtain a spacelike CMC surface in $\R^{2,1}$.
 
There is a direct way to obtain the surface $f$ from $F$:
\begin{equation}\label{sym}
f=
- \tfrac{i}{2 H} 
 \left(
 F D F^{-1} 
+ 2 \lambda (\partial_\lambda F) F^{-1} 
\right)_{\lambda=1}
\end{equation}
where $\R^{2,1}$ is regarded as the Lie algebra $\su_{1,1}$.
This is known as the Sym-Bobenko formula.
 
The discussion above is the \ll zero curvature formulation\rr of the equations for spacelike CMC surfaces in $\R^{2,1}$.  It is also the zero curvature formulation of the
equations for harmonic maps from $U_0$ to the symmetric space $SU_{1,1}/(SU_{1,1})_\si$,  where $(SU_{1,1})_\si \cong
S(U_1\times U_1)$  is the diagonal subgroup of $SU_{1,1}$.  (This symmetric space may be identified with the open unit disk in $\C$ with its hyperbolic metric.) 
Such harmonic maps may be regarded as the Gauss maps of 
spacelike CMC surfaces. The Gauss map determines the surface
up to translation in 
$\R^{2,1}$, when $H\ne 0$. 

\noindent{\em The generalized Weierstrass (DPW) representation}

The main benefit of the zero curvature formulation is that it leads directly to the local solution of the equations.  The definition of $F$ implies that it gives a holomorphic map $[F]$ from $U_0$ to 
\[
(\La SU_{1,1})_\si/(SU_{1,1})_\si,
\]
which is an open subset of the infinite-dimensional generalized flag manifold
\[
(\La SL_2\C)_\si/(\La_+ SL_2\C)_\si.
\]
Here, 
$\La_+ SL_2\C$ denotes the subset of $\La SL_2\C$ consisting of maps $S^1\to SL_2\C$ which extend holomorphically to the interior of $S^1$. 
The flag manifold has an open {\em dense} subset (the \ll big cell\rrr) represented by the similarly defined\footnote{In this article $\La_- SL_2\C$ denotes the subset of $\La SL_2\C$ consisting of maps $S^1\to SL_2\C$ which extend holomorphically to the exterior of $S^1$ in the Riemann sphere, and which are of the form $I+O(1/\la)$.}
complex affine space $(\La_-SL_2\C)_\si$.  Since $F(z_0)=I$,
on some (possibly smaller) open neighbourhood $U_0$ of $z_0$ we can write $F=F_- F_+$  (the Birkhoff factorization of $F$).
The map $F_-:U_0 \to (\La_- SL_2\C)_\si$
is a holomorphic function of $z$ which represents the
holomorphic map to the flag manifold $(\La SL_2\C)_\si/(\La_+ SL_2\C)_\si$, i.e.\ $[F_-]=[F]$. 
From
\[
F^{-1}dF = F_+^{-1}(F_-^{-1}dF_-)F_+   +
F_+^{-1} dF_+
\]
we see that the expansion of $F_-^{-1}dF_-$ contains no terms of the form $\la^i$ with $i<-1$;  on the other hand it contains only terms of the form $\la^i$ with $i\le -1$ because
$F_-=I+O(1/\la)$. Therefore,
\[
F_-^{-1}dF_-
= 
\tfrac1\la
\bp
 & p_1 \\
 p_2 & 
 \ep
 dz
\]
for some holomorphic functions  $p_1,p_2$ on $U_0$.   We call $F_-$ the  complex extended frame, and $F_-^{-1}dF_-
$ the DPW potential.  
 
Conversely, let $p_1,p_2$ be holomorphic functions on a
simply-connected open neighbourhood $U_0$ of $z_0$. 
Then there exists a unique holomorphic map 
$L:U_0\to (\La_- SL_2\C)_\si$ such that 
\[
L^{-1}dL=
\tfrac1\la
\bp
 & p_1 \\
 p_2 & 
 \ep
 dz
\]
and $L(z_0)=I$.   
Since $L(z_0)\in (\La SU_{1,1})_\si$ and 
$(\La SU_{1,1})_\si/(SU_{1,1})_\si$ is 
an open subset of 
$(\La SL_2\C)_\si/(\La_+ SL_2\C)_\si$, 
on some open neighbourhood of $z_0$ we can write $L=FB$
(the Iwasawa factorization of $L$),
where $F,B$ take values, respectively, in
$(\La SU_{1,1})_\si$, $(\La_+ SL_2\C)_\si$.
Then $F$ is an extended frame in the above sense with $L=F_-$, $B^{-1}=F_+$.  

This is the DPW generalized Weierstrass representation
for spacelike CMC surfaces in $\R^{2,1}$:  on sufficiently small domains, such surfaces correspond to pairs $(p_1,p_2)$ of holomorphic functions.  

\noindent{\em Summary}

The reader who is not familiar with surface theory (or who prefers  different conventions to those of \cite{BrRoSc10}) may regard the above discussion purely as motivation,  as we shall now summarize the formulae that will actually be used. 

For this, we need more information about the orbits of
the action of $(\La SU_{1,1})_\si$ on 
$(\La SL_2\C)_\si/(\La_+ SL_2\C)_\si$.  These orbits (see \cite{Ke99}) may be parametrized by a discrete set of elements $w$ of  $(\La SL_2\C)_\si$.  If $L$ takes values in the orbit of $w$, we have an Iwasawa factorization of the form $L=FwB$. So far we have used only the (open) orbit of $w=I$.  It turns out (section 4.5 of \cite{Ke99})
that there are precisely two open orbits, given by $w=I$ and 
\[
w=\bigw.
\]
For these two values of $w$, the Iwasawa factors $F,B$ are unique if we take $B$ to be of the form
$\diag(k,k^{-1})+O(\la)$ with $k=k(z,\zbar)>0$.  

Let us consider any smooth $(\La \su_{1,1})_\si$-valued connection form
\begin{equation}\label{GMC}
\al=
\bp
a & \tfrac1\la b \\
\tfrac1\la c & -a
\ep
dz+
\bp
-\bar a & \la \bar c \\
\la \bar b & \bar a
\ep
d\bar z,
\end{equation}
where $a, b, c$ are smooth functions which do not depend on $\la$.
This satisfies $d\al +\al\wedge\al=0$ if and only if the functions $a,b,c$ satisfy
\begin{gather*}
c_{\bar z} + 2\bar a c = 0\\
b_{\bar z} - 2\bar a b = 0\\
a_{\bar z} + \bar a_{z} - b\bar b + c \bar c = 0.
\end{gather*}
These equations may be interpreted as the Gauss-Codazzi equations of a spacelike CMC surface, in the following way.

First, let us suppose\footnote{Note that $w^{-1} (\La \su_{1,1})_\si w\sub 
(\La \su_{1,1})_\si$, so
$(Fw)^{-1} d(Fw)$ takes values in 
 $(\La \su_{1,1})_\si$, as $F^{-1}dF$ does.} that $\al$ arises as
$\al=(Fw)^{-1} d(Fw)$, for an  Iwasawa factorization
\[
L=FwB,\ 
w\in \{ I, \w \}
 \]
of an $(\La SL_2\C)_\si$-valued map 
$L$ such that
\[
L^{-1}dL =\tfrac1\la
\bp  &  p_1  \\
p_2 & 
\ep
dz,
\]
for some given holomorphic functions $p_1,p_2$.  
By comparing  coefficients of $1/\la$ in the
formula
\[
L^{-1}dL = B^{-1} (Fw)^{-1} d(Fw) B +
B^{-1} dB,
\]
we see that $b=p_1k^2$ and $c=p_2/ k^2$ (in particular, $bc=p_1p_2$), where $B=\diag(k,k^{-1})+O(\la)$ with $k>0$.
By substituting $a=\bar b_z/(2\bar b)$ and $c=p_1p_2/b$ into
$a_{\bar z} + \bar a_{z} - b\bar b + c \bar c = 0$,  we obtain
\[
(\log \vert b\vert)_{z \bar z} - \vert b\vert^2  +  \vert p_1p_2\vert^2
\vert b\vert^{-2} =0.
\]
These are the Gauss-Codazzi equations of a surface whose metric $g=4e^{2u}(dx^2 + dy^2)$, (positive) constant mean curvature $H$, and Hopf differential $Qdz^2$ satisfy
\[
\vert b\vert = e^u H,\quad
Q=2p_1p_2/H
\]
(as well as $b=p_1k^2$ and $c=p_2/ k^2$).   

In the next section we shall take $p_1=1$, $p_2=1/z$.  For this case,  we may define $u$ and $Q$ by
\[
e^u=k^2/H,\quad Q=2/(zH).
\]
(Later on, for convenience, we shall 
choose the specific value $H=\tfrac12$.)  We obtain, therefore, a specific spacelike CMC surface.

In order to match this with the previous discussion, we note that, with our choice of $u, H, Q$, the connection form (\ref{GMC}) becomes
\[
\al=
\tfrac12
\bp
u_z &  \tfrac1\la  2  H e^u \\
\tfrac1\la   Q e^{-u} & -u_z
\ep
dz
+
\tfrac12
\bp
-u_{\bar z} &  \la  \bar Q e^{-u} \\
2   \la H e^{u} & u_{\bar z}
\ep
d\bar z.
\]
This is $\diag(1,-i) (\calA dz + \calB d\bar z)\diag(1,-i)^{-1}$ where $\calA,\calB$ are given in (\ref{MC}).  
Our surface is given explicitly by replacing $F$ in the 
Sym-Bobenko formula (\ref{sym}) by 
$\diag(1,i)Fw\diag(1,i)^{-1}$. 

\begin{remark} The $1$-form $\al$ in (\ref{GMC}) does not agree exactly with that in (\ref{MC}) because various arbitrary choices were made in (\ref{MC}).   To obtain an exact match we can replace the condition \ll $k>0$\rr in the Iwasawa factorization by the condition \ll $ip_1k^2/H>0$\rrr; then we may introduce a real-valued function $u$ and a holomorphic function $Q$ by defining
$ip_1k^2=He^u$,  $Q=2p_1p_2/H$. 
\end{remark}

\section{Quantum cohomology}\label{qcoh}
 
The simplest Fano manifold is $\C P^1$, and its
 (small) quantum cohomology was one of the first examples to be computed.  With respect to the standard basis $1,b$ of $H^\ast(\C P^1;\C)\cong \C^2$, the known quantum products  $b\circ 1=b$, $b\circ b=q$ give rise to the Dubrovin/Givental connection
\[
d+\tfrac1\la
\bp
 & q\\
 1 &
 \ep \tfrac{dq}{q}
 =
 d+\tfrac1\la
\bp
 & 1\\
 1/q &
 \ep 
 dq
 \]
 in the trivial bundle $\C^\ast \times \C^2$.  

Thus, the main purpose of this article will be to investigate the harmonic map (or spacelike CMC surface) corresponding to the  DPW potential
 \[
 \eta=
\tfrac1\la
\bp
 & 1\\
1/z &
 \ep dz.
 \]
No knowledge of quantum cohomology is required for this, but we shall indicate in this section how quantum cohomology provides the appropriate
Lie-theoretic context.
 
The matrix of the Poincar\'e intersection form $(\ ,\ )$ on $H^\ast(\C P^1;\C)$  is
\[
P=
\bp  & 1\\
1 & 
\ep.
\]
The quantum product satisfies
$(a\circ b,c)=(b,a\circ c)$, which says that $\eta$  takes values in the twisted loop algebra 
$(\La \sl_2\C)_\si$, where\footnote{The expressions involving $D$ here follow from the fact that,
for $2\times 2$ matrices $A$, we have
$(A^t)^{-1} = (DP) A  (DP)^{-1}$ if $\det A=1$ and
$-A^t = (DP) A  (DP)^{-1}$ if $\tr A=0$.}
$\si$ is the involution of $\sl_2\C$ given by 
$\si(A)=-P A^t P = DAD$ and $D=\diag(1,-1)$.
The corresponding involution of $SL_2\C$ is 
$\si(A)=P (A^t)^{-1} P = D A D$.

The quantum product is weighted homogeneous, in the sense that $\vert a\circ b\vert = \vert a\vert + \vert b\vert$,  with respect to the degrees $\vert 1\vert=0, \vert b\vert=2, \vert q\vert=4$.  This is responsible for the following 
\ll homogeneity property\rr of $\eta$:
\[
\eta(\eps^2 z,\eps\la)=
\d(\eps)^{-1}\
\eta(z,\la)
\ \d(\eps),
\quad
\d(\eps)=\diag(1,\eps)
\]
for all unit complex numbers $\eps\in U_1=S^1$.

We take $U=\C^\ast=\C-\{0\}$.
To find a complex extended frame we have to solve the complex  o.d.e.\ $L^{-1}dL = \eta$.
The point $z=0$ plays an important role in quantum cohomology, but we cannot expect  $L$ to be single-valued on $\C^\ast$, and we cannot expect it to satisfy $L(0)=I$.  However, there is a canonical  solution of the form
\begin{align*}
 L&=\exp\!
 \tfrac1\la\!\!
\bp
 0 & 0\\
 \log z & 0
 \ep
 \ 
 L_0\\
 &=
 e^{tN/\la}L_0,
 \ \ \text{with}\ 
 N=\bp 0 & 0\\ 1 & 0 \ep
 \ \text{and}\  \  t=\log z,
\end{align*}
where $L_0(0)=I$.  (This $L_0$ has a natural interpretation in terms of Gromov-Witten invariants; see section 5.4 of
\cite{Gu08}.) 

An explicit formula for $L_0$ can be obtained as follows.  Since
$dL^t  (L^t)^{-1}= \eta^t$,  $L^t$ is a fundamental solution matrix for the system
\[
\la\b
\bp
\phi_0 \\ \phi_1
\ep
=
\bp 
 & 1\\
 z & 
 \ep
 \bp
\phi_0 \\ \phi_1
\ep
\]
where $\b = zd/dz$.  The equivalent second-order o.d.e.\ satisfied by $\phi_0$ is $(\la\b)^2\phi_0=z\phi_0$.  The  Frobenius method gives a natural basis $\phi,\tilde\phi$ of solutions near the regular singular point $z=0$ of this o.d.e.\ of the form
\begin{align*}
\phi&=f_0, \quad f_0(0)=1\\
\tilde\phi&=\tfrac1\la f_0\log z + f_1,\quad f_1(0)=0
\end{align*}
where
\[
f_0(z)= \sum_{i\ge 0} \frac{z^i}{(i!)^2 \la^{2i}},
\quad
f_1(z)= -\frac2\la 
\sum_{i\ge 1} (1+\cdots+ \frac1i)\frac{z^i}{(i!)^2 \la^{2i}}.
\]
This gives
\[
L=
\bp
\phi  &  \la\b \phi \\
\tilde\phi  &  \la\b\tilde\phi
\ep
\\
=
\bp
1 & 0\\
\tfrac1\la \log z & 1
\ep
\bp
f_0 &  \la\b f_0 \\
f_1  &  f_0+\la\b f_1
\ep
\ \overset{\text{def}}{=}\ 
e^{tN/\la}\,
L_0,
\]
from which the stated properties of $L_0$ follow.  Furthermore, it can be seen that the series for $L_0$ converges everywhere in $\C$.
 
In order to construct a harmonic map, we need a real form of the loop group $\La SL_2\C$.  We choose that given by the conjugate-linear involution 
\begin{equation}\label{real}
C(A)=P\bar A P = D\, ( {{\bar A}^t})^{-1}  \, D
\end{equation}
of  $SL_2\C$,
where the bar denotes complex conjugation with respect to the real form $H^\ast(\C P^1,\R)$ of $H^\ast(\C P^1,\C)$.
Thus, the relevant  loop group is $(\La SU_{1,1})_\si$, and the theory of the previous section applies. We obtain a harmonic map  to the symmetric space $SU_{1,1}/S(U_1\times U_1)$, and  a spacelike CMC surface in $\R^{2,1}$, whose Gauss map is this harmonic map. 

Now, a (metric) $tt^\ast$ structure on a domain in $\C^r$ is simply a pluriharmonic map from that domain to the symmetric space 
$GL_{2r}\R/O_{2r}$.  We refer to 
\cite{Du93}, \cite{CoSc05}, \cite{Sc05}, \cite{GuKuXX} for a full explanation of this, and in particular the relation with variations of polarized Hodge structure.  We have
$SU_{1,1}/S(U_1\times U_1)\cong SL_2\R/SO_2$ (the unit disk can be identified with the upper half plane),  and $SL_2\R/SO_2$ is a totally geodesic submanifold of $GL_2\R/O_2$, so the quantum cohomology of $\C P^1$ gives  a $tt^\ast$ structure on some domain $U\sub\C$.

The nontrivial aspect of this $tt^\ast$ structure, and the main content of this paper, concerns the nature of the domain $U$, {\em which should surround the singular point $z=0$ and be as large as possible.}  In particular, on this domain, $L$ should map into just one orbit of the Iwasawa decomposition.
The general theory (so far) does not say anything about this, as $L$ contains $\log z$ and cannot be normalized as $I$ at $z=0$. However,  Cecotti and Vafa argued
(\cite{CeVa91}, \cite{CeVa92}) that

\noindent(i)  it is possible to take $U=\C^\ast$, and 

\noindent(ii) on this domain the (multi-valued) harmonic map has the same homogeneity property as quantum cohomology.

\noindent To be precise, these apply not to $L$ but to a translate $\ga_0^{-1}L$ of $L$, where $\ga_0$ is a constant loop ($L^{-1}dL$ is unaffected by this).

We shall give precise statements and proofs of (i), (ii) in the next two sections.
Our main results are Theorem \ref{main}, which produces a family of loops $\ga_0$ (depending on a parameter $a$) such that the Iwasawa factorization $\ga_0^{-1}L=FwB$ is possible \ll locally\rrr, 
and Theorem \ref{cv}, which says that the factorization is 
possible \ll globally\rr for a certain specific value of $a$.  This value of $a$ gives the required $tt^\ast$ structure.

Let us outline here the plan of the proofs of these theorems.  
It is convenient to write
\[
\ga_0^{-1} L = EL_0,\quad 
E=\ga_0^{-1}\  e^{tN/\la}
\]
because the local problem turns out to depend only on $E$.

\noindent {\em First step:}  The Iwasawa factorization of $\X$ may be carried out easily and explicitly.  We use this to find $\ga_0$: assuming the existence of a homogeneous  Iwasawa factorization of $\X$ near $z=0$ imposes strong conditions on $\ga_0$ (Lemma \ref{h}).  For such $\ga_0$ we give the Iwasawa factorization of $\X$ in Proposition \ref{Z}.  Since $\X$ is a good approximation to $\X L_0$ near $z=0$, this allows us to deduce the existence of a homogeneous  Iwasawa factorization of $\X L_0$ near $z=0$ (Theorem \ref{main}).

\noindent {\em Second step:} To prove  the existence of a \ll global\rr homogeneous  Iwasawa factorization of $\X L_0$ for a particular value of the parameter $a$ (Theorem \ref{cv}), we appeal to a uniqueness result from the theory of Painlev\'e equations.  This argument is made possible by two fortuitous observations. First, under the homogeneity assumption,  the Gauss-Codazzi equations (of which our CMC surface is a solution) reduce to the
Painlev\'e \Rmnum3 equation. The family of solutions to this equation given by $\ga_0$ has already been studied in great detail, and it contains a certain global solution which is characterized by its asymptotic behaviour as $z\to 0$. This asymptotic behaviour is known explicitly for our family of solutions, so we can identify one of our solutions with the known global solution.

\section{First step: A family of CMC surfaces}\label{homogeneity}

Recall that $L=e^{tN/\la}L_0$ where $L_0$ is holomorphic for all $z\in\C$ and satisfies  $L_0(0)=I$, and where $t=\log z$,  $N=
\left(
\begin{smallmatrix}
0 & 0\\
1 & 0
\end{smallmatrix}
\right)
$. 

\begin{theorem}\label{main}  For any $a>0$,  consider the following element $\ga_0$ of $(\La SL_2\C)_\si$:
\[
\ga_0(\la)=
\bp
1/\sqrta   &  -\la/\sqrta  \\
0  &  \sqrta
\ep
=
\bp 1 & \\  & \la \ep^{-1}
\bp
1/\sqrta   &  -1/\sqrta  \\
0  &  \sqrta
\ep
\bp 1 & \\  & \la \ep.
\]
Then

\noindent(a) $\ga_0^{-1}L$ admits an Iwasawa factorization
\[
\ga_0^{-1}L=
F
\w
B
\]
on some domain $U=V\cap(\C-(-\infty,0])$, where $V$ is a neighbourhood of $z=0$ in $\C$, and

\noindent(b) $B$ is homogeneous, i.e.\
$
B(\eps^2 z,\eps\la)=
\d(\eps)^{-1}
B(z,\la)
\d(\eps)
$
 for all $\eps\in S^1$.
(This implies that $\al=(Fw)^{-1}d(Fw)$ is homogeneous.)

\noindent(c)
$B=\diag(k,k^{-1})+O(\la)$, where
$k$ is equal to 
$\sqrtmaa$ times a function which approaches $1$ as $z\to 0$. 
\end{theorem}

To prove the theorem, we focus on property (b), which will be sufficient to determine a family of loops $\ga_0$.  Then we shall show that properties (a) and (c) are satisfied for all such $\ga_0$. 

\begin{notation}
For any map $f=f(z,\zbar, \la)$, we shall write
$\tilde f(z,\zbar,\la)= f(\eps^2 z,\eps^{-2} \zbar,\eps\la)$. Thus, a map $f$ is homogeneous if and only if $\tilde f=\d^{-1} f \d$. To simplify notation, however, we shall omit $\zbar$  and just write $f(z,\la)$, as in the case of the smooth function $B$ in part (b) of the above theorem.
\end{notation} 

As a preliminary step, we consider 
\[
\X=\ga_0^{-1} e^{tN/\la}.
\]
This satisfies $\X^{-1}d\X=\tfrac1\la Ndt$.  Since $\tilde \X$ and $\d^{-1} \X \d$ satisfy the same equation, there exists some 
$\de=\de(\eps) \in (\La SL_2\C)_\si$ such that
\begin{equation*}
\tilde \X = \de \d^{-1} \X \d.
\end{equation*}

\begin{lemma}\label{h} ${}$

\noindent(a) If $\X$ admits an Iwasawa factorization of the form $\X=F_\X w  B_\X$ with $w=I$ or $\w$ and $\tilde B_\X = \d^{-1} B_\X \d$, then $\de \in (\La SU_{1,1})_\si$.

\noindent(b) $\de \in (\La SU_{1,1})_\si$ if and only if
\[
C(\ga_0) \ga_0^{-1} = \pm
\bp
a  &  \la \\
-1/\la & 0
\ep
\]
for some $a\in\R$.  

\noindent(c)
There exists a loop $\ga_0$ satisfying the condition  
\[
C(\ga_0) \ga_0^{-1}=
\rho
\bp
a  &  \la \\
-1/\la & 0
\ep,
\quad
\rho=\pm 1
\]
if the $(1,1)$ entry of $C(\ga_0) \ga_0^{-1}$ is positive,  i.e.\   $\rho a>0$.  When this condition holds, a suitable $\ga_0$ is given by:

(1) If $a>0$, $\rho>0$, then 
$
\ga_0=
\bp
1/\sqrta   &  -\la/\sqrta  \\
0  &  \sqrta
\ep
$

(2) If $a<0$, $\rho<0$, then 
$
\ga_0=
\bp
1/\sqrtma   &  \la/\sqrtma  \\
0  &  \sqrtma
\ep.
$
\end{lemma}

\begin{proof}
(a)  From
$
\de (\d^{-1}F_\X w \d)(\d^{-1}B_\X\d)=
\de \d^{-1} \X \d=
\tilde \X=
\tilde F_\X \tilde w \tilde B_\X
$
and the assumption
$\d^{-1} B_\X \d=\tilde B_\X$,
we have
$\de (\d^{-1}F_\X w \d)=\tilde F_\X \tilde w$,
so $\de = \tilde F_\X \tilde w \d^{-1} w^{-1} F_\X^{-1} \d$.
Now, $\tilde w \d^{-1} w^{-1}=\d^{-1}$, so
$\de=\tilde F_\X  \d^{-1} F_\X^{-1} \d 
\in (\La SU_{1,1})_\si$.

\noindent (b) We have
\begin{align*}
\de&=\tilde \X \d^{-1} \X^{-1} \d\\
&= \tilde\ga_0^{-1}  
e^{ 
(\log \eps^2z) N/(\la\eps)
}
\d^{-1}
e^{ 
-(\log z) N/\la
}
\ga_0  \d\\
&=
\tilde\ga_0^{-1} \d^{-1}
e^{ 
(\log \eps^2) N/\la
}
\ga_0  \d.
\end{align*}
Here, and elsewhere, we write $\log \eps^2 z = 
\log \eps^2 + \log z$ even though $\eps$ and $z$ are complex numbers. This is justifiable as it suffices (for our arguments) to work locally with a fixed branch of $\log$.

Thus, the required condition 
$\de \in (\La SU_{1,1})_\si$, i.e.\ 
$\de=C(\de)$, is
\begin{align*}
\tilde\ga_0^{-1} \d^{-1}
e^{ 
(\log \eps^2) N/\la
}
\ga_0  \d
&=
C(\tilde\ga_0)^{-1} 
\d^{-1}
D
e^{ 
-(\log \eps^{-2}) N^t\la
}
D\,
C(\ga_0)  \d
\\
&=
C(\tilde\ga_0)^{-1} 
\d^{-1}
e^{ 
-(\log \eps^{2}) N^t\la
}
\,
C(\ga_0)  \d,
\end{align*}
in other words
\[
C(\tilde\ga_0) \tilde\ga_0^{-1}=
\d^{-1}
\left(
e^{ 
-(\log \eps^2) N^t \la
}
C(\ga_0) \ga_0^{-1}
e^{ 
-(\log \eps^2) N/\la
}
\right)
\d.
\]
Let us assume now that $\de \in (\La SU_{1,1})_\si$.  
Putting $\la=1$ in the above equation gives
\[
C(\ga_0(\eps)) \ga_0(\eps)^{-1}=
\d^{-1}
\bp
1 & -\log \eps^2 \\
0 & 1
\ep
C(\ga_0(1)) \ga_0(1)^{-1}
\bp
1 & 0 \\
-\log \eps^2 & 1
\ep
\d.
\]
The right hand side is, a priori,  a polynomial in $\log\eps^2$.  Since $\ga_0$ is single-valued, $\log\eps^2$ cannot occur, so we must obtain the same result if we replace $\log\eps^2$ by zero.  This gives
\begin{equation}\label{oneparameter}
C(\ga_0(\eps)) \ga_0(\eps)^{-1}=
\d^{-1}
C(\ga_0(1)) \ga_0(1)^{-1}
\d,
\end{equation}
where the matrix
\[
C(\ga_0(1)) \ga_0(1)^{-1} 
\ \overset{\text{def}}{=}\ 
\bp
a & b\\
c & d
\ep
\]
satisfies
\[
\bp
a & b\\
c & d
\ep
=
\bp
1 & -\log \eps^2 \\
0 & 1
\ep
\bp
a & b\\
c & d
\ep
\bp
1 & 0 \\
-\log \eps^2 & 1
\ep;
\]
this implies that $b+c=0$ and $d=0$.  Moreover, from the fact that
$C\left(
C(\ga_0(1)) \ga_0(1)^{-1}
\right)=
\left(
C(\ga_0(1)) \ga_0(1)^{-1}
\right)^{-1}
$,
we have $a=\bar a$ and $c=-\bar b$.  Finally, since 
$\det \left( C(\ga_0(1)) \ga_0(1)^{-1} \right) = 1$, 
we must have $b=\pm 1$.  This shows that $C(\ga_0(\la)) \ga_0(\la)^{-1}$ has the stated form.

Conversely, if $C(\ga_0(\la)) \ga_0(\la)^{-1}$ has this form, 
from the explicit formula (\ref{oneparameter}) for $C(\ga_0)\ga_0^{-1}$ it is easy to verify that $C(\de)=\de$.

\noindent(c)   If $a>0$ we have
\begin{equation}\label{iwaone}
\bp
a & 1\\
-1 & 0
\ep
=
\bp
\sqrta & 0\\
-1/\sqrta & 1/\sqrta
\ep
\bp
\sqrta & 1/\sqrta\\
0 & 1/\sqrta
\ep,
\end{equation}
and if $a<0$ we have 
\begin{equation}\label{iwatwo}
-\bp
a & 1\\
-1 & 0
\ep
=
\bp
\sqrtma & 0\\
1/\sqrtma & 1/\sqrtma
\ep
\bp
\sqrtma & -1/\sqrtma\\
0 & 1/\sqrtma
\ep.
\end{equation}
Conjugating by the matrix $\diag(1,\la)$, we obtain the stated results.
\end{proof}

We can now give the Iwasawa factorization $\X=F_\X w B_\X$, which is the \ll model case\rr for our main goal, the
 Iwasawa factorization of $\X L_0$.

\begin{proposition}\label{Z} Let $\X=\ga_0^{-1} e^{tN/\la}$ where $\ga_0$ is as in part (c) of  Lemma \ref{h} with $\rho a>0$.
Then the Iwasawa factorization
$\X=F_\X w B_\X$
is given by:

\noindent(1) If $a>0$, $\rho>0$, then:

\noindent\ (1a) $w=I, 
B_\X=
\bp
\sqrtaa   &  \la/\sqrtaa  \\
0  &  1/\sqrtaa
\ep
$
if
$a+t+\bar t >0$.

\noindent\ (1b)
$w=\w\!\!, 
B_\X=
\bp
\sqrtmaa   &  -\la/\sqrtmaa  \\
0  &  1/\sqrtmaa
\ep
$
if
$a+t+\bar t <0$.

\noindent(2) If $a<0$, $\rho<0$, then:

\noindent\ (2a) $w=\w\!\!, 
B_\X=
\bp
\sqrtaa   &  \la/\sqrtaa  \\
0  &  1/\sqrtaa
\ep
$
if
$a+t+\bar t >0$.

\noindent\ (2b) $w=I, 
B_\X=
\bp
\sqrtmaa   &  -\la/\sqrtmaa  \\
0  &  1/\sqrtmaa
\ep
$
if
$a+t+\bar t <0$.

\noindent Of these,  (1b) and (2b) are valid for $z=e^t$ in
a domain $U=V\cap(\C-(-\infty,0])$, where $V$ is some neighbourhood of $z=0$ in $\C$.
\end{proposition}

\begin{proof}
Let
\begin{equation}\label{Zdef}
Z
=
C(\X)^{-1} \X 
=
C(e^{tN/\la})^{-1}
C(\ga_0) \ga_0^{-1}
e^{tN/\la}
=
\rho
\bp
a+t+\bar t  &  \la \\
-1/\la & 0
\ep,
\end{equation}
where we have used the formula for $C(\ga_0) \ga_0^{-1}$ from part (c) of Lemma \ref{h}.
We compare this with
\[
Z=C(\X)^{-1} \X = C(B_\X)^{-1} C(w)^{-1}  w B_\X,
\]
in order to read off
the expressions for $w$ and $B_\X$ (the factors 
$C(B_\X)^{-1}$,
$B_\X$
are unique
when the middle factor
$ C(w)^{-1}  w$  is $\pm I$).
Parts (1a), (2b) follow directly from 
 (\ref{iwaone}), (\ref{iwatwo}) respectively (replace
 $a$ by $a+t+\bar t$).  For part (1b), let us rewrite
 (\ref{iwatwo}) as
 \[
\bp
a & 1\\
-1 & 0
\ep
=
\bp
\sqrtma & 0\\
1/\sqrtma & 1/\sqrtma
\ep
C \w^{-1} \w
\bp
\sqrtma & -1/\sqrtma\\
0 & 1/\sqrtma
\ep.
\]
If we now replace
 $a$ by $a+t+\bar t$ we obtain the desired result.  The proof of part (2a) is similar.
\end{proof}

\begin{remark} Proposition \ref{Z} (and its proof, given by formulae (\ref{iwaone}), (\ref{iwatwo})) is concerned essentially only with the Iwasawa decomposition of the flag manifold $\C P^1$ of the finite-dimensional Lie group $SL_2\C$ with respect to its real form $SU_{1,1}$.  If we regard $\C P^1$ as the orbit of the point
$
\bsq
1\\0
\esq
$
under
$SL_2\C$, then the orbits of $SU_{1,1}$ on $\C P^1$ are the upper hemisphere, the equator, and the lower hemisphere.   If we regard 
$
\bsq
z_0\\ z_1
\esq
\in\C P^1
$
as the point $z_1/z_0$ of $\C\cup\infty$, then these three orbits are given, respectively, by the conditions $\vert z_1/z_0\vert<1$,  
$\vert z_1/z_0\vert=1$, $\vert z_1/z_0\vert>1$.  
Now, we have
\[
\X
\bsq
1\\0
\esq
=
\ga_0^{-1} e^{tN/\la}
\bsq
1\\0
\esq
=
\begin{cases}
\bsq
a+t \\
t
\esq
\ \text{in case (1)}
\\
{}
\\
\bsq
-a-t \\
t
\esq
\ \text{in case (2).}
\end{cases}
\]
In case (1),
\[
1-\left\vert \tfrac{t}{a+t} \right\vert^2  =
\tfrac{a}{\vert a+t\vert^2}  (a+t+\bar t),
\]  
so cases (1a) or (1b) correspond exactly to whether
the \ll curve\rr
\[
\X
\bsq
1\\0
\esq:
z\mapsto
\bsq
a+t
\\
t
\esq
=
\bsq
a+\log z
\\
\log z
\esq
\in
\C P^1\cong\C\cup\infty
\]
takes values in the interior or exterior of the unit disk.  In both cases the curve approaches the boundary point $1\in S^1$ as $z\to 0$.  A similar observation holds for (2a), (2b).
\end{remark}

\begin{proof}[Proof of Theorem \ref{main}]
(a) To prove the existence of an Iwasawa factorization $\X L_0=FwB$,  it suffices to  prove the existence of a Birkhoff factorization of $C(\X L_0)^{-1} \X L_0$, just as we did for 
the Iwasawa factorization of $\X$ in Proposition \ref{Z}.  Thus, we aim to find $B$ in
\[
C(L_0)^{-1}  Z   L_0 =
C(\X L_0)^{-1} \X L_0 = 
C(B)^{-1}C(w)^{-1}w B,
\]
where $Z$ is given by equation (\ref{Zdef}).

The map $Z^{-1} C(L_0)^{-1}  Z L_0$ extends continuously to $z=0$ and takes the value $I$ there, because
$L_0=I + O(z)$ and $\lim_{z\to 0} z\log \vert z\vert^2=0$.
Therefore there exists a Birkhoff factorization
\[
Z^{-1} C(L_0)^{-1}  Z L_0 = U_-U_+
\]
in a neighbourhood of $z=0$, with $U_-(0)=U_+(0)=I$.  
We obtain
\begin{align*}
C(L_0)^{-1}  Z L_0 &=Z U_-U_+\\
&= -C(B_\X)^{-1} B_\X U_- U_+
\ \text{(Proposition \ref{Z} (1b))}
\end{align*}
so it suffices to obtain a Birkhoff factorization of $ B_\X U_-$.
Now, 
\[
B_\X=
\bp
\sqrtmaa   &  -\la/\sqrtmaa  \\
0  &  1/\sqrtmaa
\ep
=
B_\X^{(1)}B_\X^{(2)},
\]
where
\[
B_\X^{(1)}=
\bp
\sqrtmaa   &  0  \\
0  &  1/\sqrtmaa
\ep,\ 
B_\X^{(2)}=
\bp
1   &  \la/
\text{\scriptsize
$(a+t+\bar t)$}
\\
0  &  1
\ep.
\]
Since $B_\X^{(2)}$
extends continuously to $z=0$ and takes the value $I$ there,
and $U_-(0)=I$, there exists a Birkhoff factorization of  $B_\X^{(2)}U_-$ near $z=0$, and hence of $B_\X U_-$.

\noindent(b)   First we show that $L_0$ is homogeneous. We use the fact that $L_0$ is determined uniquely by the o.d.e.\
\[
L_0^{-1}dL_0+ \tfrac1\la L_0^{-1} NL_0\tfrac{dz}z  =  L^{-1}dL =\tfrac1\la
\bp  &  1\\
1/z & 
\ep
dz
\]
and the initial condition $L_0(0)=I$ (the point $z=0$ is a removable singular point of this o.d.e.).  It is easy to verify that
$\d \tilde L_0 \d^{-1}$ also satisfies these conditions, so it must be equal to $L_0$.  Thus
$\tilde L_0 = \d^{-1} L_0 \d$, as required.  

Next, we have 
$\tilde \X = \de \d^{-1} \X \d$, where
(by Lemma \ref{h})  $\de\in(\La SU_{1,1})_\si$.  Hence
\[
(\tilde F \tilde w) \tilde B= \tilde \X \tilde L_0 =
\de  \d^{-1} \X \d \, \d^{-1} L_0 \d =
(\de \d^{-1} F w \d) \d^{-1} B \d, 
\]
i.e.\
\[
\d^{-1} B \d  \tilde B^{-1} = 
(\de \d^{-1} F w \d)^{-1} (\tilde F \tilde w).
\]
We have $\tilde w = \d^{-1} w \d$,  
so
$(\de \d^{-1} F w \d)^{-1} (\tilde F \tilde w)$
takes values  in 
$(\La SU_{1,1})_\si$.  Since $\d^{-1} B \d \tilde B^{-1}$ takes values in 
$(\La_+ SL_2\C)_\si$,  and its constant term is diagonal with positive entries, both sides of this equation must be $I$.  Hence $B$ is homogeneous.

\noindent(c) This follows from the formula for $B_\X^{(1)}$ in the proof of (a) above.
\end{proof}

\section{Second step: A distinguished CMC surface}\label{global}

\begin{theorem}\label{cv}  If $a=4\ga$ in Theorem \ref{main}, where
$\ga$ is the Euler constant 
$
\lim_{n\to\infty}
\left(
1+\tfrac12+\cdots+\tfrac1n - \log n
\right) 
\approx
0.577,
$
then the Iwasawa factorization of $\ga_0^{-1} L$ is 
defined on\footnote{(and on any other simply connected subset of $\C^\ast$)}
$\C - (-\infty,0]$.
\end{theorem}

In order to prove this, we focus on the CMC surface interpretation of $\ga_0^{-1} L$.  That is, we focus on the corresponding $u, Q, H$ as defined at the end of section \ref{cmc}.  We have $e^u=k^2/H$ and $Q=2/(zH)$, where $H$ is a nonzero constant (to be fixed shortly). The existence of the Iwasawa factorization is equivalent to the existence of $u$; the latter is what we shall investigate in this section.

\begin{proposition}
If the map $B$ satisfies $\tilde B=\d^{-1}B \d$, then $u$ depends only on $\vert z\vert$ (i.e.\ the metric is radially symmetric).
\end{proposition}

\begin{proof} Since $B=\diag(k,k^{-1})+O(\la)$, the condition
$\tilde B=\d^{-1}B \d$ implies that $k(\eps^2 z)=k(z)$ for all $\eps\in S^1$, i.e.\ $u(\eps^2 z)=u(z)$.
\end{proof}

Thus, 
the Gauss-Codazzi equations 
reduce
 to an o.d.e.\ in the real variable $r=\vert z\vert$:
\[
\tfrac14\left(
u_{rr} + \tfrac1r u_r
\right)
-H^2 e^{2u} + 
\tfrac{1}{
\vphantom{H^H} H^2 r^2}e^{-2u} = 0.
\]
The transformation
\[
v=u-\tfrac12 \log \vert Q\vert =
u - \tfrac12 \log\tfrac2H + \tfrac12 \log r
\]
converts this to
\[
\tfrac14\left(
v_{rr} + \tfrac1r v_r
\right)
-\tfrac{2H}r  e^{2v} + \tfrac{1}{2H r}e^{-2v} = 0.
\]
If we choose $H=\tfrac12$ we obtain
\[
v_{rr} + \tfrac1r v_r=
\tfrac8r \sinh 2v.
\]
A further transformation $x=4r^{1/2}$ converts this to the radial sinh-Gordon equation
\[
v_{xx} + \tfrac1x v_x=
2 \sinh 2v.
\]
Finally, it is easily verified that the function $y=e^v$ satisfies
\[
y^{\prime\prime} = \tfrac1y {(y^\prime)}^2 - \tfrac1x 
y^\prime + y^3 - \tfrac1y.
\]
This is the special case $\al=\be=0$, $\ga=1$, $\de=-1$ of the Painlev\'e \Rmnum3 equation
\[
y^{\prime\prime} = 
\tfrac1y {(y^\prime)}^2 - \tfrac1x y^\prime 
+\tfrac1x(\al y^2 + \be)
+ \ga y^3 + \tfrac\de y.
\]
(see section 14.4 of \cite{In26}).

\begin{corollary}\label{asymp} The family of solutions of the \P equation given by Theorem \ref{main} satisfies
the asymptotic condition
\[
y\sim -\tfrac14 x(a + 4\log x - 4\log 4)
\]
as $x\to 0$.  
\end{corollary}

\begin{proof}  Recall (from Theorem \ref{main})  that 
$k\sim \sqrtmaa
= 
\text{\scriptsize $\sqrt{-a-2\log r \vphantom{t^t} } $ }
$
as $x \ (=4r^{1/2}) \to 0$.
Since $k=\sqrt H e^{u/2}$ and $H=\tfrac12$, we have $e^u\sim -2(a+2\log r)$, hence $y=e^v = \tfrac12 e^u r^{1/2} \sim
-\tfrac14 x(a + 4\log x - 4\log 4)$.  
\end{proof}
 
\begin{proof}[Proof of Theorem \ref{cv}]
The proof depends upon the analysis in \cite{MTW77} of smooth solutions to the \P equation.  In the notation of formula  (4.121) of  that paper, 
there is a solution denoted by $y=\eta(x,0,1/\pi)$ which is smooth on $(0,\infty)$ and satisfies
$y \sim -x(\ga + \log x - \log 4)$ as $x\to 0$.  
We wish to deduce that our solution from Corollary \ref{asymp}
with $a=4\ga$ is the same as this solution, hence must extend to a (smooth) solution on $(0,\infty)$. For this we appeal to the theory of isomonodromic deformations,  and in particular the method of Chapters 13-15 of  \cite{FIKN06}.   

First, we recall that the \P equation can be regarded as an isomonodromy equation for a family of flat connections (formulae (13.1.1) and (15.1.1) of  \cite{FIKN06}). This is equivalent to the family of flat connections in sections \ref{cmc} and \ref{qcoh}, but regarded as an $r$-family of $\la$-connections rather than a 
$\la$-family of $r$-connections;  each $\la$-connection has an irregular singularity at $\la=0$ and $\la=\infty$.  
Solutions of \P are parametrized by their
monodromy data  (Stokes data) at these singular points, and
the exponential of any solution is known to be meromorphic on $\C -(-\infty,0]$.
For solutions which are real on a nonempty open subset of
$(0,\infty)$, this monodromy data
amounts  to (i) a complex number $p$,  with $1<\vert p\vert < \infty$,  or (ii) a pair $(\pm 1,a)$ where $a$ is a real number.  The second case may be regarded as the limit $p\to\infty$ of the first case.

In case (i), which is treated in detail in chapter 15 of \cite{FIKN06}, the monodromy data $p$ is equivalent to the asymptotic data $(r,s)$ at $x=0$, where
$2v(x) \sim  r \log x + s$ as $x\to 0$.  This follows from the explicit formulae (15.1.9)-(15.1.11).   Thus the asymptotic data at $x=0$ determines the solution.  
A similar result holds in case (ii), which will be treated in detail in \cite{ItNiXX}.  In case (i) one always has $\vert r\vert<2$, and case (ii) may be regarded as the limit $\vert r\vert \to 2$ ($r \to 2$ for our solution).
Case (ii) includes the above solution with asymptotics
$2v(x) \sim 2\log[-x(\ga + \log x - \log 4)]$ as $x\to 0$. We conclude that our solution from Corollary \ref{asymp}
with $a=4\ga$
agrees with the solution of \cite{MTW77}, and is therefore smooth on $(0,\infty)$.  

It follows that the frame $F$ (and the Iwasawa factorization
of $\ga_0^{-1} L_0$) is defined on any simply-connected subset of $\C^\ast$, in particular on
$\C - (-\infty,0]$.
\end{proof}

The above argument shows that the Iwasawa factorization exists on the universal covering $\tilde \C^\ast = \C$.   The image of $F$ (or the CMC surface) consists of infinitely many pieces, obtained by analytic continuation from the piece with domain $\C - (-\infty,0]$.  Images\footnote{The images shown here were made with the software XLAB created by N. Schmitt. The timelike axis of $\R^{2,1}$ lies in the plane of reflectional symmetry.}
of the piece with domain
$\{ z\in \C - (-\infty,0] \st \vert z\vert < 0.5 \}$ are shown in Figure \ref{fig1}.
\begin{figure}[ht]
\begin{center}
\includegraphics[scale=0.26]{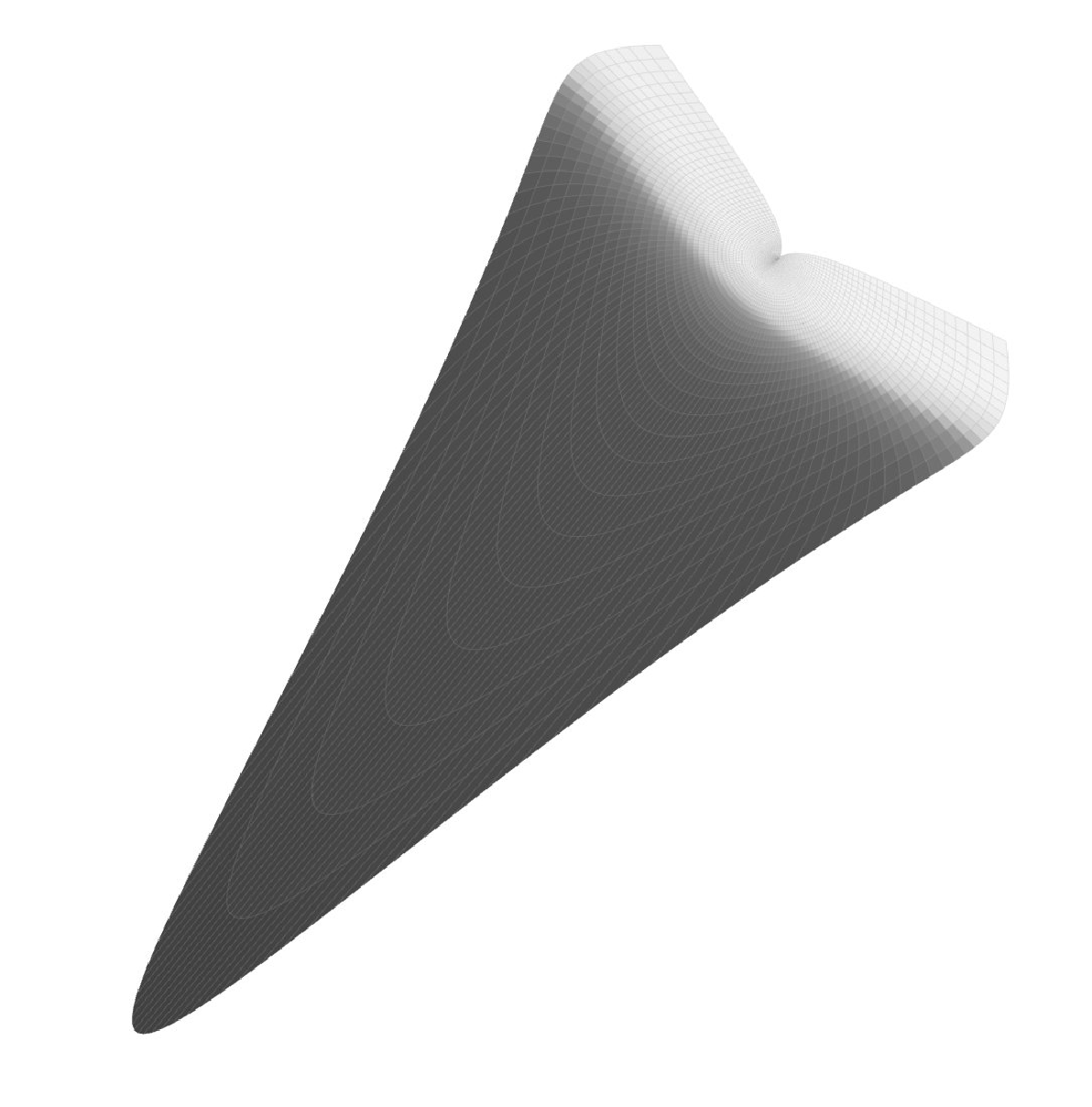}
\includegraphics[scale=0.26]{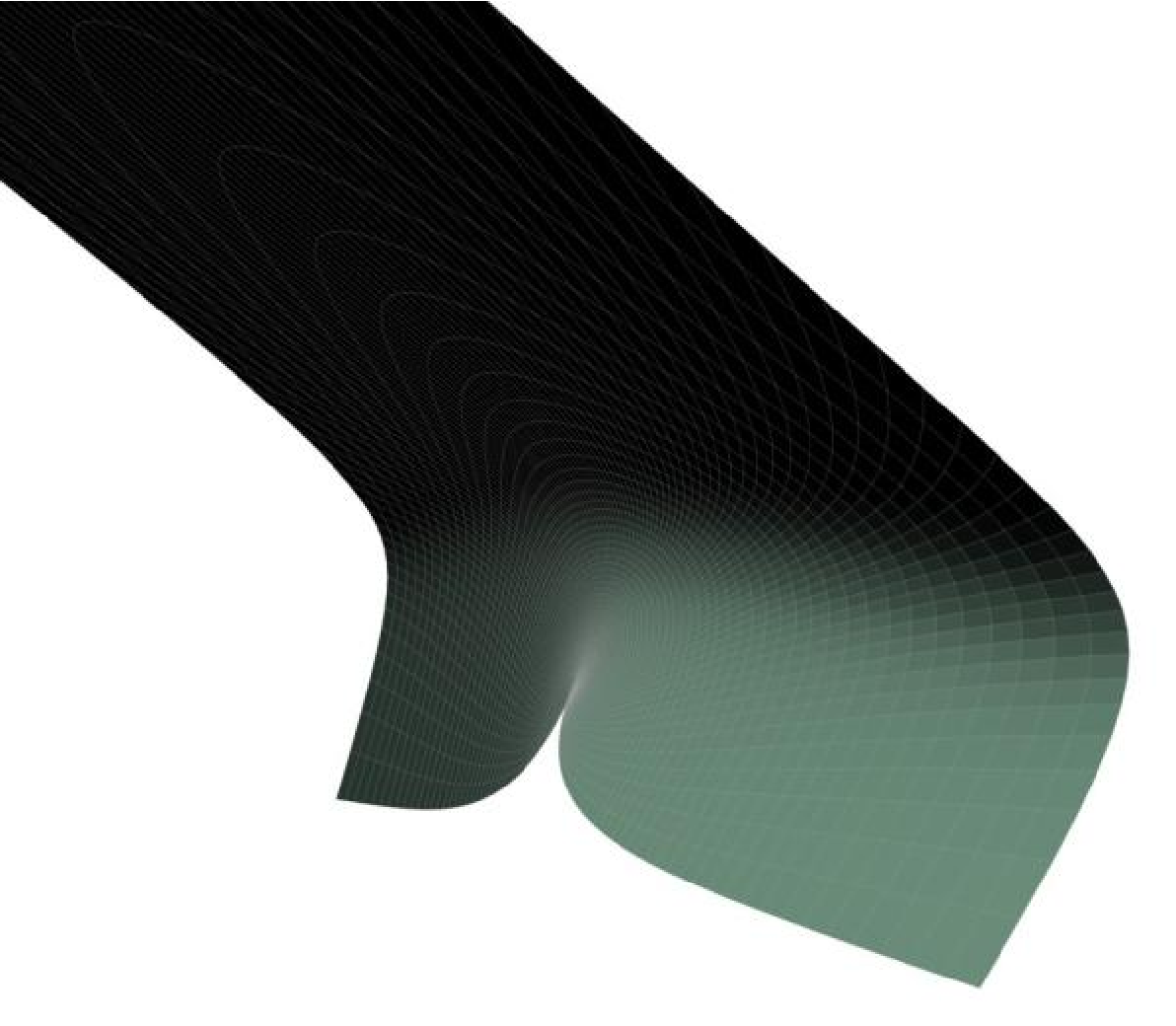}
\end{center}
\caption{Two views of the surface given by $a=4\ga$}
\label{fig1}
\end{figure}
The edges of the \ll slit\rr $(-0.5,0]$ are mapped to the short (lighter) edge at the top right of the first picture (bottom of the second picture); the singularity at the origin is clearly visible in the middle of this edge.

We conclude with some observations about this surface.

\noindent{\em Global smoothness}

This is the most important property predicted by mirror symmetry (and confirmed by our analysis): when $a=4\ga$,  the complex extended solution $\ga_0^{-1}L$  lies entirely within a single Iwasawa orbit of the loop group $\La SU_{1,1}$, and therefore the extended frame (and the surface) has no singularities.  In contrast, Figure \ref{fig2} illustrates what happens for other values of $a$.  In general one expects singularities where the complex extended solution crosses from one open Iwasawa orbit to the other.  According to Theorem 4.2 of 
\cite{BrRoSc10}, there are situations where the surface remains continuous but has cuspidal edges at the singular points.  It seems likely that the surfaces in Figure \ref{fig2} are of this type.
\begin{figure}[ht]
\begin{center}
\includegraphics[scale=0.26]{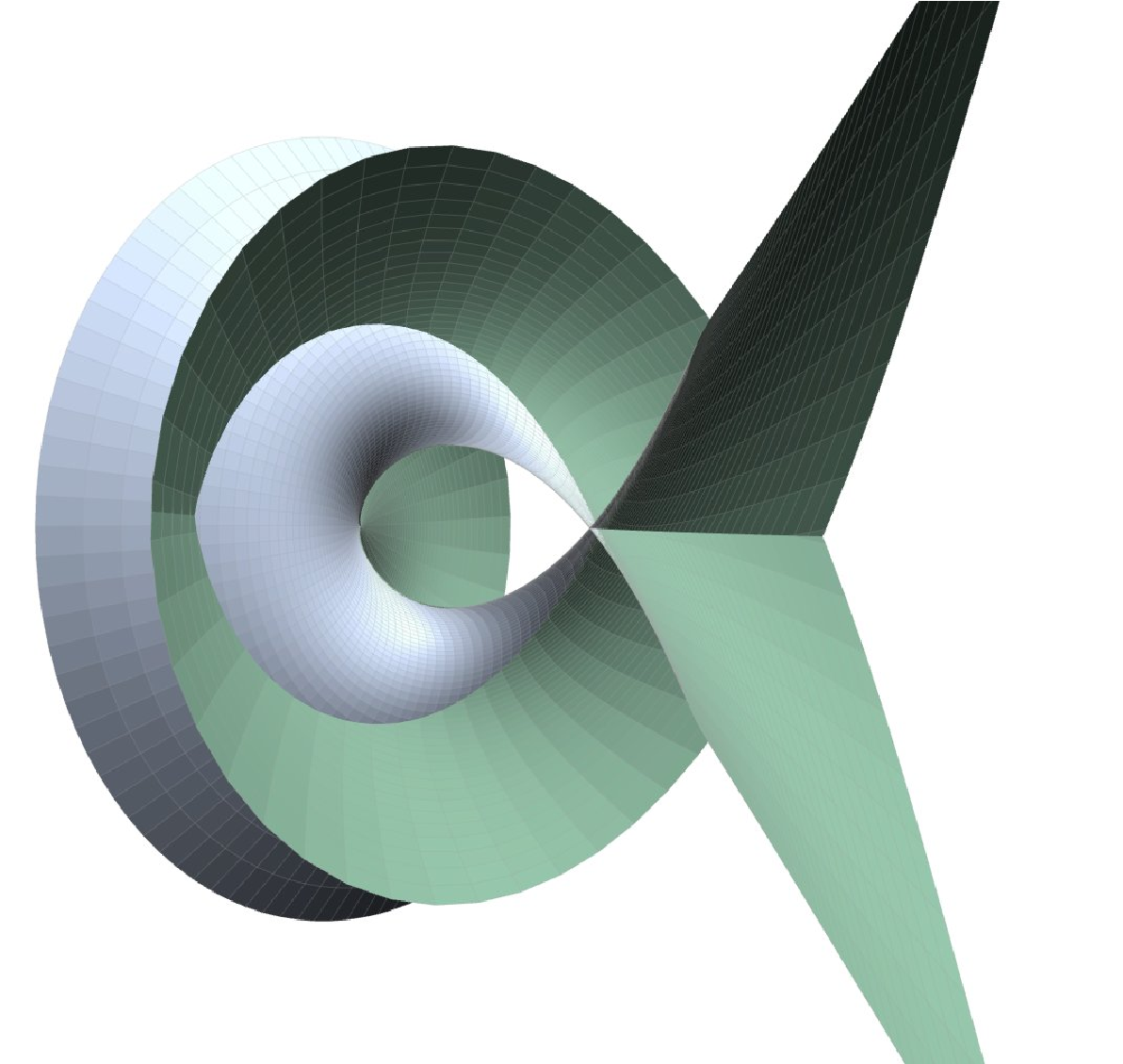}
\includegraphics[scale=0.26]{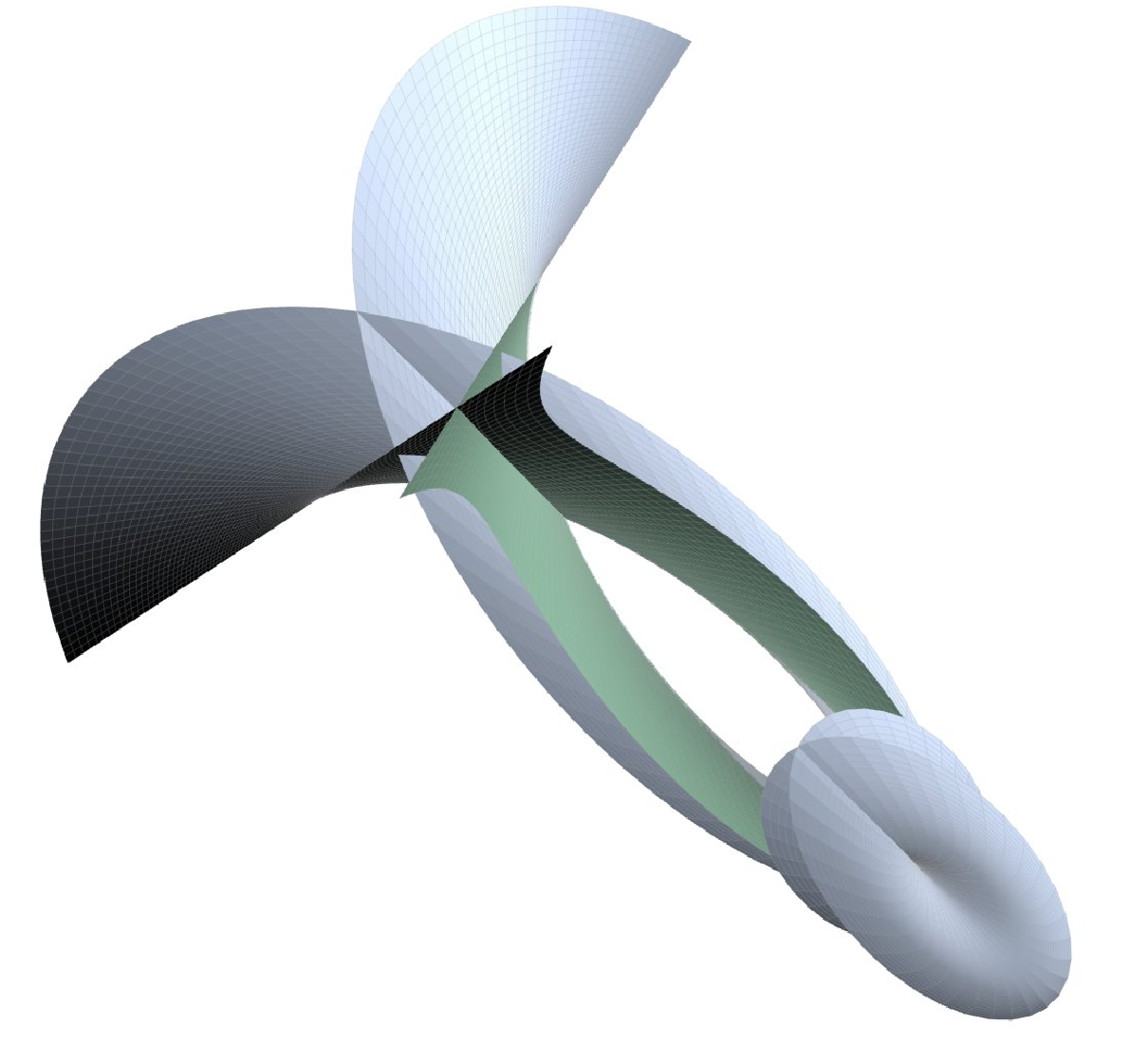}
\end{center}
\caption{The surfaces given by $a=1$ (left) and $a=4$ (right)}
\label{fig2}
\end{figure}

All of these surfaces have a reflective symmetry, due to the fact that the coefficients of $\ga_0^{-1}L$ are real, i.e.\
$\ga_0^{-1}(\la) L(z,\la) = 
\overline{\ga_0^{-1}(\bar \la) L(\bar z,\bar \la)}$.
The timelike axis lies in the plane of this reflection.

\noindent{\em Non-completeness}

From the proof of Corollary \ref{asymp}, we know that 
$e^u\sim -2(a+2\log r)$ for $r$ close to zero.  
Consider the curve in the surface parametrized by the curve $(e^{-t},0)$, $t \in (t_0,\infty)$ in the domain.  
Taking $t_0$ sufficiently large, we find that the length of the curve is 
\[ 
\int_{t_0}^\infty \sqrt{(-e^{-t})^2+0^2} 
\, 2e^u dt 
\sim 
\int_{t_0}^\infty 2e^{-t} (-2 a + 4 t) dt, 
\] 
which is finite. It follows that the surface is not complete.  

\noindent{\em Non-closing} 

The graphics (and the presence of the logarithm in $L$) suggest that none of the surfaces, for any value of $a$, are likely to be well-defined on any annulus  of the form $0< \vert z\vert < \eps$.  
In fact, it can be proved (see \cite{DoRoXX}) that no surface in the dressing orbit of the surface associated to $L$ can be well-defined on any annulus  of this form. 

\section{Related results and generalizations}\label{comments}

\noindent{\em The Grassmannian model interpretation of Iritani}

From our discussion at the end of section \ref{qcoh}, it is clear that the quantum cohomology of any manifold possesses a local $tt^\ast$ structure in a neighbourhood of any point where the map $L$ is defined (and, more generally, this holds for any Frobenius manifold).  This is an immediate consequence  of the Iwasawa decomposition.  
However, the existence of a $tt^\ast$ structure near the \ll  large radius limit point\rr  $z=0$ is not immediately obvious.  Our Theorem \ref{main} gives such structures for $\C P^1$, and the proof gives a general criterion for the existence of such structures; cf.\  the criterion given by Iritani (Theorem 1.1 and Proposition 3.5 of \cite{IrXX}).

Iritani uses the infinite dimensional \ll Grassmannian model\rr of the homogeneous space
$(\La SL_2\C)_\si/(\La_+ SL_2\C)_\si$.  As mentioned in the introduction, this exhibits the $tt^\ast$ structure as an infinite dimensional generalization of a variation of polarized Hodge structure.  The Grassmannian (or, more precisely, flag manifold) is constructed using \ll semi-infinite\rr subspaces of the Hilbert space $H^{(2)} = L^2(S^1,\C^2)$.  Iritani considers real structures on this complex Hilbert space, given by conjugate-linear involutions 
$\kappa:H^{(2)}\to H^{(2)}$ satisfying certain properties.  The involution singled out by Iritani (using K-theory) in section 5 of \cite{IrXX} is given by
\[
\kappa_{\calH}(x)=
\bp
\la & 0 \\
-4\ga & -1/\la
\ep
\bar x
\]
on elements $x\in H^{(2)}$.
This induces the following involution on the loop group $(\La SL_2\C)_\si$:
\begin{align*}
\kappa_{\calH}(A)(\la)&=
\bp
\la & 0 \\
-4\ga & -1/\la
\ep\bar A(1/\bar\la)
\bp
\la & 0 \\
-4\ga & -1/\la
\ep^{-1}
\\
&=
\bp
\la & 0 \\
-4\ga & -1/\la
\ep
P\,C(A)(\la)P
\bp
\la & 0 \\
-4\ga & -1/\la
\ep^{-1}\!\!\!\!\!,
\ \ P=
\bp
0 & 1 \\
1 & 0
\ep
\\
&=
\bp
4\ga & \la \\
  -1/\la & 0
\ep^{-1} 
C(A)(\la)
\bp
4\ga & \la \\
  -1/\la & 0
\ep
\end{align*}
where $C$ is the involution (given earlier in formula (\ref{real})) which defines the real form  $(\La SU_{1,1})_\si$.  The involution $\kappa^\tau_{\calH}$ of \cite{IrXX}
 is then given in our notation by 
$\kappa^\tau_{\calH}=Z^{-1} C Z$, where $Z$ was defined in formula (\ref{Zdef}). Thus, Iritani's modification of the standard real structure $C$ is equivalent to our modification of  $L$ by $\ga_0$.

\bigskip
\noindent{\em CMC surfaces in $\R^3$}

Although they do not arise from quantum cohomology, the surfaces with holomorphic data $p_1=1$, $p_2=z^n$, $n\in\Z$,  are natural generalizations of the case $n=-1$ that we have studied here, and we shall discuss them elsewhere. 
We just make some brief comments here
on the analogous CMC surfaces in $\R^3$, i.e.\ with the same holomorphic data but using $SU_2$ instead of $SU_{1,1}$.

For $n\ge 0$, these are the well known Smyth surfaces, and they are also \ll globally smooth\rr (see \cite{Sm93}). 
It was shown in \cite{BoIt95} that the Iwasawa factorization method greatly simplifies the original differential equation arguments of \cite{Sm93}.   In fact no differential equation arguments are needed at all:
the Iwasawa decomposition has only one orbit,  as $SU_2$ is compact,  so the Iwasawa factorization $L=FB$ is possible on the whole domain of $L$.  Moreover, the frame and surface are smooth at $z=0$.  It follows that the holomorphic data 
$p_1=1$, $p_2=z^n$ gives (for $n\ge 0$) a  CMC surface in $\R^3$ which is globally smooth on $\C$.  

Bobenko and Its point out that the Gauss-Codazzi equations in this case reduce to a version of the sinh-Gordon equation ($w_{z\bar z}=-\sinh w$  instead of $w_{z\bar z}=\sinh w$),  whose radially invariant reduction is again a special case of the \PP\  equation.  They use the Iwasawa factorization to deduce explicit connection formulae for this \PP\  equation.  In contrast, in the case $SU_{1,1}$, we must use results on the \PP\  equation to show that our Iwasawa factorization can be carried out globally.  

The CMC surfaces in $\R^3$ with $n<0$ have so far not been considered by differential geometers. However, it is interesting to note that, in the case $n=-1$, it is not possible to find any $\ga_0$ for which the analogue of  Theorem \ref{main} holds (the conditions on the coefficients $a,b,c,d$ of the matrix in the proof of Lemma \ref{h} cannot be satisfied).

\bigskip
\noindent{\em The quantum cohomology of $T^2$}

There is one other manifold whose quantum cohomology gives rise to a spacelike CMC surface in $\R^{2,1}$,  namely the torus $T^2=S^1\times S^1$.  The appropriate quantum cohomology algebra is that part of the small quantum cohomology algebra which is generated by $H^2 T^2$,  and
this is simply $\C[b,q]/(b^2)$ (the quantum product is equal to the cup product).  The quantum differential operator is simply the operator $(\la \b)^2$.  This gives the DPW potential
\[
\tfrac1\la N \tfrac{dq}q=
\tfrac1\la
\bp
0 & 0\\
1 &  0
 \ep \tfrac{dq}q,
 \]
which has canonical extended solution $L=e^{tN/\la}$.
The calculations of section \ref{homogeneity} apply to this case, but are much easier.  Apart from the very simple nature of $L$, there are two further simplifying factors: (i) the homogeneity condition $\tilde f=T^{-1} f T$ is vacuous, as
$T(\eps)=\diag(1,1)$, and (ii) the generalized Weierstrass representation of section \ref{cmc} does not apply directly when $H=0$, but it can be replaced by the ordinary Weierstrass representation.
(Spacelike surfaces in $\R^{2,1}$ with $H=0$ are called maximal surfaces.  A Weierstrass representation for such surfaces was given in \cite{Ko83}; this is analogous to the usual one for minimal surfaces, i.e.\ surfaces  in $\R^3$ with $H=0$.  For treatments of minimal surfaces in the loop group context we refer to \cite{DoPeTo97} and section 4.5 of \cite{Gu08}.  Maximal surfaces may be dealt with in the same way.)

A CMC surface is determined by its Gauss map only when $H\ne 0$, so the  quantum cohomology of $T^2$ does not determine a canonical maximal surface, but rather the set of all such surfaces which have the  Gauss map 
$z\mapsto e^{tN}
\bsq
1\\0
\esq
=
\bsq
1\\ \log z
\esq
$.  
This is a holomorphic map into $\C P^1$, but it does {\em not} define a \ll global\rr $tt^\ast$ structure because the image of the map is not contained in the unit disk.  

There is another aspect of the quantum cohomology of $T^2$, which does lead to a $tt^\ast$ structure, and this should be regarded as the correct analogue of what we did for
the quantum cohomology of $\C P^1$.
Namely, because of the mirror symmetry phenomenon,
the quantum differential operator $(\la q\b/\b q)^2$ arises from the operator
\[
(\la\b)^2 - 3z\la^2(3\b + 2)(3\b + 1),\quad
\b=z\tfrac{\b}{\b z}
\]
after a certain change of variable $q=q(z)$ known as the mirror transformation.  This can be analyzed in the same way as the operator  $(\la q\b/\b q)^2 - q$  (in the latter case the mirror transformation is trivial: $q=z$).  We shall just state the results here, referring to Examples 10.11, 10.14, 10.16
 of \cite{Gu08} 
(and the references listed there) for the details.  

First, the DPW potential turns out to be
\[
\tfrac1\la
\bp
0 & 0\\
\b(\tilde\phi/\phi) &  0
 \ep \tfrac{dz}z,
 \]
where $\phi,\tilde\phi$ are the natural solutions given by the Frobenius method of the above o.d.e.\ in a neighbourhood of the regular singular point $z=0$.  We have
\begin{align*}
\phi&=f_0,  \ \ f_0(0)=1,\\
\tilde\phi&=\tfrac1\la f_0\log z + f_1, \ \  f_1(0)=0,
\end{align*}
where $f_0,f_1$ are holomorphic in a neighbourhood of  $z=0$.  This gives
\[
L=
\bp
\phi  &  \la\b \phi \\
\tilde\phi  &  \la\b\tilde\phi
\ep
\\
=
\bp
1 & 0\\
\tfrac1\la \log z & 1
\ep
\bp
f_0 &  \la\b f_0 \\
f_1  &  f_0+\la\b f_1
\ep
\ \overset{\text{def}}{=}\ 
e^{tN/\la}\,
L_0,
\]
which is similar to the case of $\C P^1$, except that the formulae for $f_0,f_1$  (hence $L_0$) are different.  The mirror transformation which converts back to the trivial potential 
$\tfrac1\la N \tfrac{dq}q$ is $q=e^{\tilde\phi(z)/\phi(z)}$.

The Gauss map of any associated maximal surface is the 
multi-valued holomorphic map $\tilde\phi/\phi$, and it is well known that the image of this map is the unit disk (actually a hemisphere of the Riemann sphere). Thus we obtain a $tt^\ast$ structure defined globally on the universal cover of $\C^\ast$. This is the well known variation of Hodge structure of an elliptic curve.

\bigskip
\noindent{\em The quantum cohomology of projective spaces and weighted projective spaces}

According to Cecotti and Vafa, the quantum cohomology of any complex projective space $\C P^n$ is expected to behave in a similar way to the case $n=1$; it gives a solution of the radially symmetric affine Toda equations (a system of $n$ ordinary differential equations).
Our method certainly applies to this case, and more generally to  the orbifold quantum cohomology of any weighted complex projective space $P(w_0,\dots,w_n)$, and in every case it gives a local $tt^\ast$ structure near $z=0$.  However, we do not know\footnote{We have mentioned the recent results of \cite{GuLiXX} for the cases $\C P^3,\C P^4$ in an earlier footnote. As will be shown in a future publication, the methods of \cite{GuLiXX} can, in fact, be extended to verify the prediction of Cecotti and Vafa for any $\C P^n$.} of any global result for $\C P^n$, except  in the case $n=2$ where the o.d.e.\ reduces to a \P equation as in the case $n=1$.  
The harmonic map given by $\C P^n$ or $P(w_0,\dots,w_n)$ can be described more appropriately as a primitive map into a $k$-symmetric space (for some $k$).  It is well known  (see, for example, chapter 21 of \cite{Gu97}) that such maps correspond to solutions of affine Toda equations.  However, the particular solutions which arise here apparently have not been considered by differential geometers.

{\em

\noindent
Zentrum Mathematik  \newline
Technische Universit{\"a}t M{\"u}nchen \newline
Boltzmannstrasse 3\newline
D-85747 Garching\newline
GERMANY

\noindent
Department of Mathematics and Information Sciences\newline
   Faculty of Science and Engineering\newline
   Tokyo Metropolitan University\newline
   Minami-Ohsawa 1-1, Hachioji, Tokyo 192-0397\newline
   JAPAN

\noindent
Department of Mathematics\newline
   Faculty of Science\newline
   Kobe University\newline
   Rokko, Kobe 657-8501\newline
   JAPAN

}

\end{document}